\documentclass[11pt, reqno]{amsart}

\makeatletter
\def\theequation{\@arabic\c@equation}
\usepackage{amsmath,amsthm,amscd,amssymb,latexsym,upref}
\usepackage{tikz}
\usepackage{color}
\usepackage{versions}
\usepackage{enumerate}
\usepackage{graphicx}

\usetikzlibrary{matrix}

\numberwithin{equation}{section}

\newcommand{\we}{\wedge}

\newcommand{\Dd}{\mathbb{D}}
\newcommand{\Tt}{\mathbb{T}}

\newcommand{\Cc}{\mathbb{C}}

\newcommand{\telwe}{\dot{\wedge}}

\newcommand{\eiu}{e^{i\theta}}

\DeclareMathOperator\ran{ran}
\DeclareMathOperator\spn{span}

\DeclareMathOperator\p{P}

\DeclareMathOperator\Pls{PLS}
\DeclareMathOperator\Poc{POC}

\def\be{\begin{equation}}
\def\ee{\end{equation}}

\def\s0{s_0}
\def\p0{p_0}

\newtheorem{theorem}{Theorem}[section]

\newtheorem{lemma}[theorem]{Lemma}
\newtheorem{proposition}[theorem]{Proposition}

\newtheorem{definition}[theorem]{Definition}
\newtheorem{remark}[theorem]{Remark}
\newtheorem{example}[theorem]{Example}
\newtheorem{fact*}{Fact}

\newcommand\half{{\tfrac 12}}

\newcommand\id{{\mathrm{id}}}

\newcommand{\T}{\mathbb{T}}

\newcommand{\C}{\mathbb{C}}

\newcommand{\ip}[2]{\left\langle #1, #2 \right\rangle}

\newcommand{\inv}{^{-1}}

\renewcommand\phi{\varphi}

\newcommand\la{\lambda}

\newcommand\beq{\begin{equation}}

\newcommand\eeq{\end{equation}}

\newcommand\blue{\color{blue}}
\newcommand\black{\color{black}}

\newcommand\bbm{\begin{bmatrix}}
	\newcommand\ebm{\end{bmatrix}}
\newcommand\bpm{\begin{pmatrix}}
	\newcommand\epm{\end{pmatrix}}
\numberwithin{equation}{section}

\theoremstyle{definition}

\begin{document}
	\title[Pointwise creation operators] {Exterior powers and pointwise creation operators}

	\author{ Dimitrios Chiotis}
	\address{School of Mathematics,  Statistics and Physics, Newcastle University, Newcastle upon Tyne
		NE\textup{1} \textup{7}RU, U.K.}
	\email{chiotisd@gmail.com}
	\author{Zinaida A. Lykova}
	\address{School of Mathematics,  Statistics and Physics, Newcastle University, Newcastle upon Tyne
		NE\textup{1} \textup{7}RU, U.K.}
	\email{Zinaida.Lykova@ncl.ac.uk}
	
	\author{ N. J. Young}
	\address{School of Mathematics, Statistics and Physics, Newcastle University, Newcastle upon Tyne NE1 7RU, U.K.
		{\em and} School of Mathematics, Leeds University,  Leeds LS2 9JT, U.K.}
	\email{Nicholas.Young@ncl.ac.uk}


	\date{7th October 2020}
	\@namedef{subjclassname@2020}{\textup{2020} Mathematics Subject Classification}
	

\thanks{\blue N.J. Young \black is the corresponding author.}

	\begin{abstract}  We develop a theory of pointwise wedge products of vector-valued functions on the circle and the disc, and 
obtain results which give rise to a new approach to the analysis of the matricial Nehari problem.
We investigate properties of pointwise creation operators and pointwise orthogonal complements in the context of operator theory and the study of vector-valued analytic functions on the unit disc.\\

	\end{abstract} 

	\maketitle

	\sloppy
	\fussy
	
	\pagenumbering{arabic}
	\setcounter{page}{1}

\section{Introduction}\label{intro}
 The wedge product of Hilbert spaces, though a long established theory (see, for example \cite{Sim1,Sim2,depillis,Greub}), deserves in our view to be better exploited in operator theory than it has been hitherto.  In this paper we put forward a new approach to some aspects of the analysis of $E$-valued functions on the unit disc $\Dd$  and on the unit circle $\Tt$, where $E$ is a separable Hilbert space. Our first application of this idea was to the problem of the superoptimal analytic approximation of a continuous matrix-valued function on the unit circle \cite{YCL2020}. The computation of such approximants arises naturally in the context of the classical ``Nehari problem", and also in its application to the ``robust stabilization problem" in control engineering.  In \cite{YCL2020} we give a new algorithm for the construction of the unique superoptimal analytic approximant of a continuous
matrix-valued function on the unit circle, making use of exterior powers of operators in preference to spectral or  Wiener-Masani factorizations.
This algorithm is parallel to the construction of \cite{Constr}, but has the advantage that it
requires only the spectral factorisation of \emph{scalar} functions on the circle, together with the calculation of singular-value decompositions, and otherwise requires only rational arithmetic.  In the algorithm we make use of pointwise creation operators and pointwise orthogonal complements. 

We show how the Hilbert space geometry of the Hardy space $H^2(\Dd, E)$, where $E$ is a separable Hilbert space, interacts with the pointwise geometry of   $H^2(\Dd, E)$, that is, the geometry of $E$ over each point of $\Tt$ or $\Dd$. We work with  the completion of the algebraic tensor product of Hilbert spaces, namely the Hilbert tensor product (see \cite{Dixmier} and \cite{Tensor}), and, in particular, with its closed linear subspace of antisymmetric tensors.  The space of all antisymmetric tensors $\we^p E $, also called a {\it wedge} or {\it exterior product}, is a closed linear subspace of the $p$-fold Hilbert tensor product $\otimes_H^pE$, see Section \ref{exterior}.

In Section \ref{point_w_p}, we define the pointwise wedge product of maps that are defined either on the unit disc or on the unit circle and take values in Hilbert spaces. We illustrate the fact that  some classical results from function theory extend to the pointwise wedge product. An exemplary result is Proposition \ref{a.19} which asserts that, for any separable Hilbert space $E$,  the pointwise wedge product of two functions in the Hardy space $H^2(\Dd,E)$ is an element of $H^1(\Dd,\we^2E).$  We study pointwise creation operators, which are of the form 
$$C_\xi\colon H^2(\Dd,E)\to H^2(\Dd,\we^2E), f\mapsto \xi\telwe f,$$ 
where $\xi\colon \Dd \to E$ is a bounded analytic function, and the symbol $\xi\telwe f$ denotes the function
\[
(\xi \telwe f)(z)=\xi(z)\we f(z) \mbox{ for all }z \in \Dd.
\]
Analogous notation applies when $\xi, f$ are $E$-valued functions defined almost everywhere on $\Tt$.
  We connect the kernel of $C_\xi$ to the pointwise linear span and pointwise orthogonal complement of $\xi.$ 

\begin{definition}\label{H-infty} Let $E$ be a separable Hilbert space.
 $H^{\infty}(\Dd,E)$ denotes the space of bounded analytic $E$-valued functions $Q$ on the unit disk  with supremum norm: 
  $$\|Q\|_{H^{\infty}} \stackrel{\emph{def}}{=} \|Q\|_{\infty} \stackrel{\emph{def}}{=} \sup\limits_{z \in \mathbb{D}}\|Q(z)\|_{E}.$$
 \end{definition}

 In Section 4 we prove, among other things, the following two statements.\\

\noindent {\bf Theorem}  \ref{xjclosed}. 
{\it 
Let $E$ be a separable Hilbert space, let  $\xi_0, \xi_1, \cdots, \xi_j \in H^\infty(\Dd, E)$. Suppose that the set $\{\xi_i(z)\}_{i=0}^{j}$ is  orthonormal in $E$ for almost every $z \in \Tt.$ Then 
 $$ 
\xi_0 \telwe \dots \telwe \xi_j \telwe H^2(\Dd,E)
$$
is a closed subspace of $H^2(\Dd,\we^{j+2}E).$
}

Let $F$ and $E$ be Hilbert spaces. The Banach space of bounded linear operators from $F$ to $E$ with the operator norm is denoted by $\mathcal{B}(F,E)$.

\begin{definition}\label{xistar}
Let $E$ be a separable Hilbert space and let $\xi, \eta \in L^\infty(\T,E)$.
We define $\xi \eta^* \in L^\infty(\T,\mathcal{B}(E,E))$ by 
\[
(\xi \eta^*)(z) x =\langle x, \eta(z) \rangle \ \xi(z) \mbox{ for all } x \in E \ \mbox{and for almost every } \ z \ \mbox{on} \ \T.
\]
\end{definition}

Thus $\xi \eta^*(z)$ is the operator of rank one on $E$ that is sometimes denoted by $\xi(z) \otimes \eta(z)$ (for example, in \cite[equation (1.17)]{AMY}).

\noindent {\bf Theorem}  \ref{cx*cx}.
{\it
Let $E$ be a separable Hilbert space. Let $\xi \in H^\infty(\Dd,E)$ be an inner function. Then, for any $h \in H^2(\Dd,E),$
$$C_\xi^* C_\xi h =  P_{+} \alpha_\xi, $$ where $\alpha_\xi= h - \xi \xi^* h$ and $P_+ \colon L^2(\Tt,E)\to H^2(\Dd,E)$ is the orthogonal projection. Moreover 
$$ C_\xi^* C_\xi h = h -T_{\xi\xi^*}h,$$ where $T_{\xi\xi^*}\colon H^2(\Dd,E)\to H^2(\Dd,E)$ is the  Toeplitz operator with symbol $\xi\xi^*$. 
}

 Thus $C_\xi^*C_\xi$ is  the Toeplitz operator with symbol $1-\xi\xi^*.$ Finally, we show that although $C_\xi$ is  an isometry from the pointwise orthogonal complement of $\xi$ to $H^2(\Dd,\we^2E),$ a simple example proves that $C_\xi$ fails to be a partial isometry.

\section{Exterior powers of Hilbert spaces and operators}\label{exterior}

In this Section  we recall the established notion of the exterior power, or wedge product, of Hilbert spaces. One can find definitions and properties of wedge products in  \cite{depillis}, \cite{Greub}, \cite{pavan}, \cite{Sim1,Sim2}
 and \cite{Wed}.
Here we present a concise version of this theory which we need for the development of  pointwise wedge products of vector-valued functions on the circle and the disc.
We shall assume the notion of the algebraic tensor product of linear spaces as given, for example, in \cite{Tensor}.

For $E$ a Hilbert space, we consider the completion of the p-fold algebraic tensor product, denoted by $\otimes_H^pE,$ with respect to the norm induced by the inner product \eqref{tensorinner}. We define an action of permutation operators on tensor products and examine various properties. 
These permutation operators generate two different types of tensors, symmetric and antisymmetric, and we focus on the properties of the antisymmetric tensors. The space of all antisymmetric tensors, also called wedge or exterior products, in $\otimes_H^pE$ is denoted by $\we^p E$ and in Theorem \ref{a.10} we prove it is a closed linear subspace of $\otimes_H^pE.$  

In the following $E$ denotes a Hilbert space.

\begin{definition}
$\otimes^{p}E$ is the $p$-fold algebraic
tensor product 
of $E,$ spanned by tensors of the form $x_1 \otimes x_2 \otimes \dots \otimes x_p,$ where $ x_j \in E$ for
$j=1, \dots,  p.$
\end{definition}

\begin{definition} An inner product on $\otimes^{p}E$
 is given  on elementary tensors by
\begin{equation}\label{tensorinner}
\langle x_1 \otimes x_2 \otimes \dots \otimes x_p, 
y_1 \otimes y_2 \otimes \dots \otimes y_p \rangle_{\otimes^{p}E}=p!\langle x_1 , y_1\rangle_E \cdots \langle x_p, y_p\rangle_E
\end{equation}
for any $x_1, \dots, x_p, y_1, \dots, y_p \in E$,
and is extended to 
$\otimes^p E$ by sesqui-linearity. 
\end{definition}

\begin{definition}
$\otimes_H^p E$ is the completion of $\otimes^p E$ with respect to the norm\linebreak $\|u\| = \langle u, u\rangle^{1/2}_{\otimes^{p}E},$ for $u \in \otimes^p E.$\end{definition}

Observe that the inner product \eqref{tensorinner}, in contrast to the majority of the ones included in the bibliography, invokes a multiple of $p!$\;. The reason for this choice will be apparent in Theorem \ref{a.6}.

In order to introduce antisymmetric tensors, we need to consider the action of the following permutation operators on tensors.

\begin{definition}\label{sigma}
Let $\mathfrak{S}_p$ denote the \index{symmetric group}symmetric group on $\{1,\dots,p\},$ with the operation of composition.
 For $\sigma \in \mathfrak{S}_p$, we define 
\[
S_\sigma \colon \otimes^p E \to \otimes^p E
\]
 on elementary tensors by 
$$\displaystyle S_\sigma(x_1 \otimes x_2 \otimes \dots \otimes x_p)=x_{\sigma(1)} \otimes x_{\sigma(2)} \otimes \dots \otimes x_{\sigma(p)},$$
and we extend $S_\sigma$ to  $\otimes^p E$ by linearity, that is, for $u= \sum_{i=1}^n \la_i x_1^i \otimes \cdots \otimes x_p^i,$ we define 
$$
S_\sigma (u) = \sum\limits_{i=1}^{n} \la_i S_\sigma (x_1^i \otimes \cdots \otimes x_p^i ).
$$
for any $x_j^i \in E$ and $\la_i\in\C$.
\end{definition}

\begin{remark}
 $(\mathfrak{S}_p, \circ)$ is a group, and so, for every permutation $\sigma \in \mathfrak{S}_p,$ there exists $\sigma^{-1} \in \mathfrak{S}_p$ such that 
 $$\sigma \circ \sigma^{-1} = \id = \sigma^{-1} \circ \sigma,$$ where $id\in \mathfrak{S}_p$ is the identity map on $\{1,\dots,p\}.$ 

\end{remark}

Elements of $\mathfrak{S}_p$ induce unitary operators on $\otimes^p_H E$ in an obvious way.
\begin{proposition}\label{a.3}
Let $E$ be a Hilbert space, and let $p$ be a positive integer. Then, for any $\sigma\in \mathfrak{S}_p$,
 $S_\sigma$ is a linear operator on the normed space $(\otimes^p E , \| \cdot \|)$, which extends to an isometry $\mathbf{S}_\sigma$ on $(\otimes^p_H E,\| \cdot \|) $. Furthermore, $\mathbf{S}_\sigma$ is a unitary operator on $\otimes_H^p E $. 
\end{proposition}

\begin{proof}
It is easy to check that $S_\sigma$ is linear.
For any elementary tensors $w=x_1 \otimes x_2 \otimes \dots \otimes x_p,$\linebreak
 $v=y_1\otimes y_2\otimes \cdots y_p$  by the definition of the inner product on $\otimes^pE,$
\begin{align*}
\ip{S_{\sigma}^*w}{v}_{\otimes^p E} &= \ip{w}{y_{\sigma(1)}\otimes \dots \otimes y_{\sigma(p)}}_{\otimes^p E} \\
	&=p! \prod\limits_{j=1}^p \ip{x_j}{y_{\sigma(j)}}_E  \\
	&=p! \prod\limits_{j=1}^p \ip{x_{\sigma\inv(j)}}{y_j}_E \\
	&=\ip{S_{\sigma\inv}w}{v}_{\otimes^p E} .
\end{align*}
Hence $S_\sigma^*= S_{\sigma\inv}$, thus
\[
S_\sigma^*S_\sigma = S_{\sigma\inv} S_\sigma = I,
\]
the identity operator on $\otimes^p E$,
and therefore  $S_\sigma$ is an isometric linear self-map of  $\otimes^pE$.   Likewise, $S_\sigma S_\sigma^* = I$, and so $S_\sigma$ is also a surjective self-map of $\otimes^p E$.

Thus one can extend $S_\sigma$ by continuity to an isometric linear self-map $\mathbf{S}_\sigma$ of the completion $\otimes_H^p E$ of $\otimes^p E$.  Since
$\mathbf{S}_\sigma$ is isometric, its range is complete, hence closed in $\otimes_H^p E$.  Since the range of $\mathbf{S}_\sigma$ contains that of $S_\sigma$, $\ran \mathbf{S}_\sigma = \otimes_H^p E$. Being both surjective and isometric, 
 $\mathbf{S}_\sigma$ is a unitary operator on $\otimes^p_H E$.
\end{proof}
Henceforth we shall denote the extended operator $\mathbf{S}_\sigma$ by $S_\sigma$.

\begin{definition}\label{a.4}
A tensor $u \in \otimes_H^{p}E$ is said to be \emph{symmetric} if $S_\sigma(u)=u$ for all $\sigma \in \mathfrak{S}_p.$ 
\index{symmetric tensor}
A tensor $u \in \otimes_H^{p}E$ is said to be \emph{antisymmetric} if $u=\epsilon_{\sigma}S_\sigma u$ for all $\sigma \in \mathfrak{S}_p,$
where $\epsilon_{\sigma}$ is the signature of $\sigma.$
\end{definition}

Note that $\epsilon_{\sigma \circ \sigma^{-1}} = \epsilon_\sigma \epsilon_{\sigma^{-1}} =1, $ and hence $\epsilon_\sigma=  \epsilon_{\sigma^{-1}}.$
 
\begin{definition}
The space of all antisymmetric tensors in $\otimes_H^{p}E$ will be denoted by $\wedge^p E$. 
\end{definition}

\begin{theorem}\label{a.10}
Let $E$ be a Hilbert space. Then $\wedge^{p}E$ is a closed linear subspace of the Hilbert space $\otimes_H^{p}E$ for any $p \geq 2.$
\end{theorem}\vspace{2ex}
\begin{proof}
 For $\sigma \in \mathfrak{S}_p$ define the operator 
 $$
f_\sigma \stackrel{\emph{def}}{=} S_\sigma - \epsilon_\sigma  I\;\colon \otimes_H^p E \to \otimes_H^p E,
$$
 where $I$ denotes the identity operator on $ \otimes_H^p E $.  Since $S_\sigma$ is a continuous linear operator on $\otimes_H^pE$, $f_\sigma$ is a continuous linear operator. 
 The kernel of the operator $f_\sigma$ is 
 $$
 \begin{array}{cllllllllll}
 \ker f_\sigma &=\{ u \in \otimes_H^p E \colon (S_\sigma - \epsilon_\sigma I)(u)= 0\} \\
 &=\{ u \in \otimes_H^pE \colon S_\sigma (u) = \epsilon_\sigma u  \}\\
 &= \{ u \in \otimes_H^pE \colon \epsilon_\sigma S_\sigma(u)= u\}.
  \end{array}
$$
 
\noindent Since $f_\sigma$ is a continuous linear operator on $\otimes_H^p E,$ $\ker f_\sigma$ is a closed linear subspace of $\otimes_H^p E.$ Thus $\we^p E$ is a closed linear subspace of $\otimes_H^p E,$
 since 
\[
\we^p E = \{u \in \otimes_H^p E \; \colon\epsilon_\sigma S_\sigma(u) = u \mbox{ for all } \sigma \in \mathfrak{S}_p\} = \bigcap\limits_{\sigma \in \mathfrak{S}_p} \ker f_\sigma. 
 \qedhere\]\end{proof}

Theorem \ref{a.10} implies that the orthogonal projection from $\otimes_H^p E$
onto $\we^pE$ is well-defined.
 \begin{definition}\label{a.5}
Let $E$ be a Hilbert space.  For $x_1, \dots, x_p \in E,$ define $x_1 \wedge x_2 \wedge \dots \wedge x_p$ to be the orthogonal
  projection  of the \index{elementary tensor}elementary tensor $x_1 \otimes x_2 \otimes \dots \otimes x_p$ onto $\wedge^p E$, that is
  $$x_1 \wedge x_2 \wedge \dots \wedge x_p= P_{\we^p E}( x_1 \otimes \cdots \otimes x_p).$$
   \end{definition}

 \begin{theorem}\label{a.6}
  For all $u \in \otimes_H^p E,$  
  $$P_{\wedge^p E}(u) = \displaystyle\frac{1}{p!} \sum\limits_{\sigma \in \mathfrak{S}_p}\epsilon_{\sigma}S_\sigma (u).$$
 \end{theorem}
 \begin{proof}
  Let $u \in \otimes_H^p E$. Then, for any $\sigma \in \mathfrak{S}_p,$
   $u = \epsilon_{\sigma}S_\sigma (u) + (u - \epsilon_{\sigma}S_\sigma (u)),$ and so
   $$p!u = \sum\limits_{\sigma \in \mathfrak{S}_p} \epsilon_{\sigma}S_\sigma (u) + \sum\limits_{\sigma \in \mathfrak{S}_p}( u - \epsilon_{\sigma}S_\sigma (u)).$$ 
It suffices to show that 
 $\sum\limits_{\sigma \in \mathfrak{S}_p} \epsilon_{\sigma}S_\sigma (u)\in \we^pE $ and
$$ \sum\limits_{\sigma \in \mathfrak{S}_p}(u - \epsilon_{\sigma}S_\sigma (u))$$ is 
  orthogonal to the set of antisymmetric tensors, in other words, that if $v \in \wedge^p E$ then 
  $$\langle v , \sum\limits_{\sigma \in \mathfrak{S}_p} (u - \epsilon_{\sigma}S_\sigma (u))\rangle_{\otimes^p_H E} = 0.$$
Let $w= \sum\limits_{\sigma \in \mathfrak{S}_p} \epsilon_\sigma S_\sigma(u)\in \otimes_H^pE$. For every $\tau \in \mathfrak{S}_p,$ we have 
  $$\begin{array}{clll}
  \epsilon_\tau S_\tau (w) &=  \displaystyle\epsilon_\tau S_\tau \left(  \sum\limits_{\sigma \in \mathfrak{S}_p} \epsilon_\sigma S_\sigma(u) \right) \vspace{2ex}
= \displaystyle\sum\limits_{\tau \circ \sigma \in \mathfrak{S}_p} \epsilon_{\tau \circ \sigma} S_{\tau\circ \sigma} (u) \vspace{2ex}\\
  &= \displaystyle\sum\limits_{\sigma' \in \mathfrak{S}_p} \epsilon_{\sigma'} S_{\sigma'} (u) \vspace{2ex}= w,\end{array}$$where $\tau \circ \sigma = \sigma'.$ 
Hence $\sum\limits_{\sigma \in \mathfrak{S}_p} \epsilon_{\sigma}S_\sigma (u)\in \we^pE .$
  
  \noindent For every $v \in \we^p E,$ we have $v = \epsilon_\sigma S_\sigma v $ for all $\sigma \in \mathfrak{S}_p,$ and
\[  \begin{array}{cllllllllllll} 
  \langle v , \sum\limits_{\sigma \in \mathfrak{S}_p}(u -\epsilon_{\sigma}S_\sigma (u))\rangle_{\otimes_H^p E}&=  \sum\limits_{\sigma \in \mathfrak{S}_p} \langle v , u \rangle_{\otimes_H^p E} - 
  \sum\limits_{\sigma \in \mathfrak{S}_p} \epsilon_{\sigma}\langle v , S_\sigma (u) \rangle_{\otimes_H^p E} \vspace{2ex} \\ 
  &=  \sum\limits_{\sigma \in \mathfrak{S}_p} \langle v , u \rangle_{\otimes_H^p E} - \sum\limits_{\sigma \in \mathfrak{S}_p} \epsilon_{{\sigma}} \langle {S_\sigma}^* v , u \rangle_{\otimes_H^p E} \vspace{2ex} \\ 
  &= \sum\limits_{\sigma \in \mathfrak{S}_p}  \langle v - \epsilon_{{\sigma}^{-1}}{S_{\sigma^{-1}}} v , u \rangle_{\otimes_H^p E} \vspace{2ex} =0.
  \end{array} \]\end{proof}

 \begin{remark}
 	If $p>1,$ then $\mathfrak{S}_p$ contains a transposition, for instance $\sigma=(1\;2),$ and $\epsilon_\sigma=-1.$ If $p=1,$ then $\we^1E = E.$ 
 \end{remark}
 \begin{proposition}\label{a.8}
Let $E$ be a Hilbert space and let $p\geq 2$. The set of antisymmetric tensors and the set of symmetric tensors are orthogonal in $\otimes_H^p E.$
 \end{proposition}
 \begin{proof}
 Suppose that $u$ is a symmetric tensor, that is $S_\sigma u= u$ for all $\sigma \in \mathfrak{S}_p$, and that $v$ is an antisymmetric tensor, that is, $ S_\sigma v = \epsilon_\sigma  v$  for all $\sigma \in \mathfrak{S}_p$.
Since $ p\geq 2$ there exists a transposition $\sigma\in \mathfrak{S}_p$, so that $\epsilon_\sigma=-1$, and therefore $S_\sigma v=-v$. By Proposition \ref{a.3}, $S_\sigma$ is a unitary operator on $\otimes_H^p E$. Thus 
 $$\langle u , v \rangle_{\otimes_H^p E} = \langle S_\sigma u , S_\sigma v \rangle_{\otimes_H^p E} 
 = \langle u, - v \rangle_{\otimes_H^p E}.$$
 Thus 
 $\langle u , v \rangle_{\otimes_H^p E} = - \langle u , v \rangle_{\otimes_H^p E},$ and so
 $\langle u , v\rangle_{\otimes_H^p E} =0.$   
\end{proof}

\begin{proposition}\label{we}
Let $E$ be a Hilbert space.
 The inner product in $\wedge^p E$ is given by
 
 $$\langle x_1 \wedge \dots \wedge x_p, y_1 \wedge \dots \wedge y_p \rangle_{\wedge^p E} = \det \begin{pmatrix}
                                                                                   
                                                                                   \langle x_1, y_1 \rangle_E &\dots & \langle x_1, y_p \rangle_E\\
                                                                                   					\vdots &\ddots & \vdots \\
                                                                                   \langle x_p, y_1  \rangle_E & \dots &\langle x_p, y_p \rangle_E
                                                                                                                                                                                                     \end{pmatrix}$$
for all $x_1, \dots, x_p, y_1, \dots, y_p \in E$.
\end{proposition}

\begin{proof}
 By Theorem \ref{a.6}, we have
\[ \begin{array}{cllllllllllllllllll}
 &\langle x_1 \wedge \dots \wedge x_p, y_1 \wedge \dots \wedge y_p \rangle_{\wedge^p E} \\ [5ex]
 &=\left\langle \displaystyle\frac{1}{p!}\sum\limits_{\sigma \in \mathfrak{S}_p}\epsilon_{\sigma}S_\sigma (x_1 \otimes x_2 \otimes \dots \otimes x_p) , \displaystyle\frac{1}{p!}\sum\limits_{\tau \in \mathfrak{S}_p}\epsilon_{\tau}S_\tau 
 (y_1 \otimes y_2 \otimes \dots \otimes y_p) \right\rangle_{\otimes_H^p E} \vspace{2ex}\\ 
 &=\displaystyle\frac{1}{p!^2} \sum\limits_{\sigma,\tau \in \mathfrak{S}_p} \langle \epsilon_\sigma S_\sigma (x_1 \otimes x_2 \otimes \dots \otimes x_p), \epsilon_\tau S_\tau 
 (y_1 \otimes y_2 \otimes \dots \otimes y_p) \rangle_ {\otimes_H^p E} \vspace{2ex}\\ 
 &=\;\displaystyle \frac{1}{p!^2}  \sum\limits_{\sigma,\tau \in \mathfrak{S}_p}\epsilon_{\sigma}\epsilon_{\tau} \langle  x_1 \otimes x_2 \otimes \dots \otimes x_p,S_\sigma^* S_\tau  (y_1 \otimes y_2 \otimes \dots \otimes y_p)  \rangle_{\otimes_H^p E} \vspace{2ex}\\ 
  &=\;\displaystyle \frac{1}{p!^2} \sum\limits_{\sigma,\tau \in \mathfrak{S}_p}\epsilon_{\sigma^{-1}}\epsilon_{\tau} \langle x_1 \otimes x_2 \otimes \dots \otimes x_p,S_{\sigma^{-1}} S_\tau (y_1 \otimes y_2 \otimes \dots \otimes y_p ) \rangle_ {\otimes_H^p E}\vspace{2ex} \\ 
  &=\;\displaystyle \frac{1}{p!} \sum\limits_{\sigma' \in \mathfrak{S}_p}\epsilon_{\sigma'}  \langle  x_1 \otimes x_2 \otimes \dots \otimes x_p,S_{\sigma'} (y_1 \otimes y_2 \otimes \dots \otimes y_p )  \rangle_ {\otimes_H^p E}\vspace{2ex} \\ 
  &= \;\displaystyle  \sum\limits_{\sigma'\in \mathfrak{S}_p}\epsilon_{\sigma'} \prod\limits_{i=1}^p \langle  x_{i}, y_{\sigma'(i)}\rangle_E \vspace{2ex} \\
   &= \det \begin{pmatrix}
           \langle x_1 , y_1 \rangle_E & \cdots & \langle x_1 , y_p\rangle_E\\
           \vdots &\ddots & \vdots\\
           \langle x_p , y_1 \rangle_E & \cdots & \langle x_p , y_p \rangle_E
                    \end{pmatrix}
  \end{array}\]
by Leibniz' formula.\end{proof}
Since we have already shown that $\we^p E$ is a closed linear subspace of the Hilbert space $\otimes_H^pE,$ the space $(\we^p E, \langle\cdot,\cdot\rangle_{\we^pE})$ with inner product given by Proposition \ref{we} is itself a Hilbert space. 

\begin{lemma}\label{weon}
	Suppose $\{u_1, \cdots, u_j\}$ is an orthonormal set in $E.$ Then, 
for every $x \in E,$
	$$\| u_1  \we \cdots \we u_j \we x\|_{\we^{j+1}E} =\| x - \displaystyle\sum\limits_{i=1}^j \langle x, u_i  \rangle u_i\|_{E}.   $$
\end{lemma}
\begin{proof} For $x \in E$ we may write
	$$ x = x -  \displaystyle\sum\limits_{i=1}^j \langle x, u_i \rangle u_i + \displaystyle\sum\limits_{i=1}^j \langle x, u_i  \rangle u_i.$$  
	\noindent By Proposition \ref{we}, 
		$$ 	\begin{array}{lllllllll}
	&\| u_1 \we \cdots \we u_j \we x\|_{\we^{j+1}E}^2  \vspace{3ex}\\ &= \langle u_1 \we \cdots \we u_j \we x , u_1 \we \cdots \we u_j \we x \rangle_{\we^{j+1}E} \vspace{3ex}\\
	&= \det\begin{pmatrix}
	\langle u_1, u_1\rangle_{E}  & \langle u_1, u_2\rangle_{E}&\cdots&\cdots& \langle u_1 ,x\rangle_{E} \\
	\langle u_2, u_1 \rangle_{E} & \langle u_2 , u_2 \rangle_{E} &\cdots &\cdots& \langle u_2, x \rangle_{E} \\
	\vdots &\cdots & \ddots &\cdots&\cdots \\
\langle u_j , u_1 \rangle_{E} & \langle u_j , u_2 \rangle_{E} &  \cdots &\langle u_j, u_j \rangle_{E}& \langle u_j , x \rangle_{E}\\
	\langle x , u_1 \rangle_{E} & \langle x , u_2 \rangle_{E} &\cdots &\cdots&\langle x,x\rangle_{E}
	\end{pmatrix}.	
	\end{array}$$  
By assumption,
	$$\langle u_i, u_k \rangle =\left\{ \begin{array}{ll}
	0,\quad \mbox{ if } i\neq k\\
	1,\quad \mbox{ if }  i=k \end{array},\right.$$ and hence	
	$$\begin{array}{lllllllll}
	&\| u_1  \we \cdots \we u_j \we x\|_{\we^{j+1}E}^2  = \det\begin{pmatrix}
	1  & 0 &\cdots& \langle u_1 , x\rangle_{E} \\
	0 & 1 &\cdots & \langle u_2, x \rangle_{E}\\
	\vdots &~ & \hspace{-9ex}\ddots &\vdots \\
	0&\cdots & 1 & \langle u_j,x\rangle_{ E}\\ 
	\langle x , u_1 \rangle_{E} &\langle x , u_2 \rangle_{E} &\cdots &\langle x,x \rangle_{E}
	\end{pmatrix}.	
	 \end{array}$$ 
\noindent If, for $k=1, \cdots, j$ we multiply the $k$-th column of the determinant by $ \langle u_k , x \rangle_{E}$ and subtract it from the $(j+1)$-th column, we find that
$$\begin{array}{clll}
\| u_1  \we \cdots \we u_j \we x\|_{\we^{j+1}E}^2 
&= \det\begin{pmatrix}
1  & 0 &~&\cdots& 0 \\
0 & 1 &~&\cdots & 0\\
\vdots &~ & \ddots &~&\vdots \\%
0&~&\cdots & 1 & 0\\ 
\langle x , u_1 \rangle_{E} &\langle x , u_2 \rangle_{E}& \cdots &\cdots &\langle x,x\rangle_{E} - \sum\limits_{i=1}^j |\langle x, u_i  \rangle_{E}|^2 
\end{pmatrix}\vspace{3ex}\\
	&=\| x\|_{E}^2-  \displaystyle\sum\limits_{i=1}^j |\langle x, u_i  \rangle_{E}|^2 \vspace{3ex}\\
	&= \| x - \displaystyle\sum\limits_{i=1}^j \langle x, u_i\rangle_{E} u_i \|_{E}^2,
\end{array} $$
the latter equality by Pythagoras' theorem. 	
\end{proof}

\begin{definition}
	Let $(E, \| \cdot\|_E) $ be a Hilbert space. The \emph{$p$-fold Cartesian product of $E$} is defined to be the set 
	$$ \underbrace{E\times \dots \times E}_{p-times} =\{ (x_1,\dots, x_p): x_i \in E \}.$$ Moreover, we define a norm on  $\underbrace{E\times \dots \times E}_{p-times}$ by 
	$$\|(x_1,\dots, x_p) \|=\{\sum_{i=1}^p \|x_i\|_E^2\}^\half. $$
\end{definition}

\begin{definition}
Let $E$ be a Hilbert space. We define the multilinear operator $$\Lambda \colon \underbrace{E\times \dots \times E}_{p-times} \to \we^p E$$ by 
$$\Lambda(x_1,\dots,x_p)= x_1 \we\dots \we x_p \quad \text{for all}\quad x_1,\dots,x_p \in E. $$

\end{definition}

\begin{proposition}\label{Hadam}{\rm [Hadamard's inequality, \cite{Sing}, p. 477]}
	For any matrix $$A=(a_{ij})\in \Cc^{n\times n},$$
	$$|\det(A)| \leq \prod\limits_{j=1}^{n}\left( \sum\limits_{i=1}^n |a_{ij} |^2  \right)^{1/2} \quad \text{and} \quad |\det(A)| \leq \prod\limits_{i=1}^{n}\left( \sum\limits_{j=1}^n |a_{ij} |^2  \right)^{1/2}. $$
\end{proposition}

\begin{proposition}\label{weopiscontinuous1}
Let $E$ be a Hilbert space. Then the multilinear mapping $$\Lambda \colon \underbrace{E\times \dots \times E}_{p-times} \to \we^p E$$ is bounded.
\end{proposition}

\begin{proof}
 Let $x_i \in E$ for all $i=1,\dots,p.$ Then $ \Lambda(x_1,\dots,x_p ) = x_1 \we \dots \we x_p$ and 
$$\begin{array}{clll} \| \Lambda(x_1,\dots,x_p )\|_{\we^p E}^2 &= \| x_1\we \dots \we x_p \|_{\we^p E}^2\vspace{2ex}\\
&= \langle x_1\we \dots \we x_p, x_1\we \dots \we x_p \rangle_{\we^p E} \vspace{2ex} \\
&=\det \begin{pmatrix}
\langle x_1, x_1 \rangle_E & \langle x_1,x_2 \rangle_E&\dots  &\langle x_1,x_p \rangle_E \\ 
\langle x_2 , x_1 \rangle_E & \langle x_2,x_2 \rangle_E & \dots & \langle x_2 , x_p \rangle_E \\
\vdots & \vdots & \ddots & \dots \\
\langle x_p, x_1 \rangle_E & \dots &\dots& \langle x_p, x_p \rangle_E 
\end{pmatrix}\geq 0.  \end{array} \label{deteq}$$ 
Let 
$$X= \begin{pmatrix}
\langle x_1, x_1 \rangle_E & \langle x_1,x_2 \rangle_E&\dots  &\langle x_1,x_p \rangle_E \\ 
\langle x_2 , x_1 \rangle_E & \langle x_2,x_2 \rangle_E & \dots & \langle x_2 , x_p \rangle_E \\
\vdots & \vdots & \ddots & \dots \\
\langle x_p, x_1 \rangle_E & \dots &\dots& \langle x_p, x_p \rangle_E 
\end{pmatrix}. $$ 
By Hadamard's inequality,
$$|\det(X)| \leq \prod\limits_{j=1}^{p}\left( \sum\limits_{i=1}^p |\langle x_i, x_j \rangle_E |^2  \right)^{1/2}.
$$

\noindent Moreover, by the Cauchy-Schwarz inequality,
$$|\det(X)| \leq \prod\limits_{j=1}^p \|x_j\|_E \left( \sum\limits_{i=1}^p \|x_i\|_E^2 \right)^{1/2}.
$$
Therefore
\begin{equation}\label{Had-C-S-inq}
\| \Lambda(x_1,\dots,x_p )\|_{\we^p E}^2 \leq  \prod\limits_{j=1}^p \|x_j\|_E \left( \sum\limits_{i=1}^p \|x_i\|_E^2 \right)^{1/2}.
\end{equation}
Let $\|(x_1,\dots,x_p)\|_{E^p} \leq 1$.  Since $\|x_j\|_E\leq \|(x_1,\dots,x_p)\|_{E^p} \leq 1$ for each $j$, we have
\begin{align*}
\| \Lambda (x_1, \dots, x_p)  \|^2_{\we^p E}  &  \leq 1.
\end{align*}
Hence the $p$-linear operator $\Lambda$ is bounded.  
\end{proof}

\section{Pointwise wedge products}\label{point_w_p}

In this section we introduce the notion of pointwise wedge product of vector-valued functions on  the unit circle or in the unit disk and explore its features.

\begin{definition}\label{a.17}
Let $E$ be a Hilbert space and let $f,g\colon \Dd \to E$  $\mathrm{(} f,g\colon \Tt \to E \mathrm{)}$ be $E$-valued maps. We define the \emph{pointwise wedge product of $f$ and $g,$}
$$f\telwe g \colon \Dd \to \we^2E \quad \mathrm{(} f\telwe g \colon \Tt \to \we^2E\mathrm{)}$$ 
by $$(f\telwe g) (z) = f(z) \we g(z) \quad \text{for all}\; z \in \Dd \quad \mathrm{(} \text{for almost all}\;  z \in \Tt\mathrm{)}.$$

\end{definition}\vspace{3ex}
 \begin{definition}\label{pointwiseld}
   Let $E$ be a Hilbert space and let $\chi_1,\dots,\chi_n \colon \mathbb{D} \to E$ $\mathrm{(}\chi_1,\dots,\chi_n \colon \mathbb{T} \to E \mathrm{)}$ be $E$-valued maps. We call $\chi_1,\dots \chi_n$ \emph{pointwise linearly dependent} on $\Dd$ \emph{(}or  on $\Tt$\emph{)} if  for all $z\in\Dd$ 
\rm{(}for almost all $z\in\Tt$ respectively\rm{)} the vectors $\chi_1(z), \dots,\chi_n(z)$ are linearly dependent in $E$.

   \end{definition}

\begin{remark}
 If $x_1, \dots, x_n$ are pointwise linearly dependent on $\Tt$, then 
 $$(x_1\telwe \dots \telwe x_n)(z)=0$$ for almost all $z \in \Tt.$
\end{remark}

\subsection{Pointwise wedge products on function spaces}

For vector-valued $L^p$ spaces we use the terminology of \cite{NagyFoias}.
\begin{definition}\label{a.12}	
Let $E$ be a separable Hilbert space and let $1\leq p < \infty.$  Define
\begin{enumerate}
 \item[{\rm (i)}] $L^p (\Tt,E)$ to be the normed space of measurable 
(weakly or strongly, which amounts to the same thing, in view of the separability of $E$)
 $E$-valued maps  $f \colon \Tt \to E$ such that 
 $$\|f\|_p = \left(\displaystyle\frac{1}{2\pi}\int_{0}^{2\pi} \|f(e^{i\theta})\|_E^p d\theta\right)^{1/p} <\infty ;$$
      \item[{\rm (ii)}] 
     $H^p (\mathbb{D},E)$ to be the normed space of analytic $E$-valued maps $f\colon\mathbb{D} \to E $ such that
     $$ \|f\|_p = \sup\limits_{0<r<1} \left(\frac{1}{2\pi}\int_{0}^{2\pi} \|f(re^{i\theta})\|_E^p d\theta\right)^{1/p} < \infty;$$
\item[{\rm (iii)}] 
$L^{\infty}(\mathbb{T},E )$ is the space of essentially bounded  measurable $E$-valued functions on the unit circle with essential supremum norm
 $$\|f\|_{L^\infty}= \mathrm{ess} \sup\limits_{|z|=1}\|f(z)\|_{E}.$$ 
    \end{enumerate}
   \end{definition}   
\begin{proposition}\label{a.18}
Let $E$ be a separable Hilbert space and let 
$\displaystyle\frac{1}{p}+\frac{1}{q}=1,$ where $1 \leq p,q \leq \infty$. Suppose that
 $x\in L^p(\mathbb{T},E),$ $y \in L^q(\mathbb{T},E).$ Then
 $$
x \dot{\we} y \in L^1(\mathbb{T},\we^2 E).
$$
and 
\be\label{Hwedge}
 \|x \dot{\we} y\|_{ L^1(\mathbb{T},\we^2 E)} \leq  \|x\|_{L^p(\T,E)}  \|y\|_{L^q(\T,E)}.
\ee
\end{proposition}
\begin{proof}
 By Proposition \ref{we}, for all $z \in \Tt,$
 $$\begin{array}{cllllllll}
    \|(x \dot{\wedge} y)(z)\|_{\we^2 E}^2 &= \langle x(z) \wedge y(z), x(z) \wedge y(z) \rangle_{\we^2 E}\\
 &= \langle x(z) ,x(z) \rangle_{E} \cdot \langle y(z),y(z)\rangle_{E} - |\langle x(z) ,y(z) \rangle_{E}|^2 \\
 &\leq\|x(z)\|_{E}^2  \|y(z)\|_{E}^2.
    \end{array}$$ Thus, for all $z \in \Tt,$
   $$\label{wenorm} \|(x\dot{\wedge} y)(z)\|_{\we^2 E} \leq  \|x(z)\|_{E} \|y(z)\|_{E}. $$
 By Definition \ref{a.12}, 
\begin{equation}\label{eqwe} \| x\telwe y \|_{L^1(\Tt, \we^2 E)} = \displaystyle\frac{1}{2\pi} \int\limits_0^{2\pi} \|(x\dot{\wedge} y)(e^{i\theta})\|_{\we^2 E} \;d\theta \leq 
\frac{1}{2\pi} \int\limits_0^{2\pi}\|x(e^{i\theta})\|_E \|y(e^{i\theta})\|_E \; d\theta.\end{equation}     
     
\noindent Now by H\"{o}lder's inequality, 
\begin{equation}\label{hl} \frac{1}{2\pi}\int\limits_0^{2\pi} \|x(e^{i\theta})\|_E \|y(e^{i\theta})\|_E \;d\theta \leq  \left(\frac{1}{2\pi}\int\limits_0^{2\pi} \|x(e^{i\theta})\|_E^p \; d\theta \right)^{1/p} 
\left(\frac{1}{2\pi}\int\limits_0^{2\pi} \|y(e^{i\theta})\|_E^q \; d\theta \right)^{1/q}. \end{equation} 
                                                                           
\noindent By inequalities (\ref{eqwe}) and (\ref{hl}), $x\dot{\wedge} y \in L^1(\mathbb{T},\we^2E)$ and the inequality \eqref{Hwedge} holds.
\end{proof}
\begin{proposition}\label{wejanalytic}
Let $E$ be a Hilbert space and $x_1, x_2, \dots, x_n \colon \Dd \to E$ be  analytic $E$-valued maps on $\Dd.$
Then, $$x_1 \telwe x_2\telwe \dots \telwe x_n\colon \Dd \to \we^n E $$ is also analytic on $\Dd$ and 
$$(x_1 \telwe x_2 \telwe \dots \telwe x_n)'(z) = x_1'(z)\we x_2(z) \we \dots \we x_n(z)  + \dots + x_1(z) \we x_2(z)\we \dots \we x_n'(z)$$ for all $z\in \Dd.$
\end{proposition}
 The proof is straightforward. It follows from Proposition \ref{we} and the continuity of $\Lambda$, see Hadamard's inequalities \eqref{Had-C-S-inq}.

\begin{proposition}\label{a.19}
Let $E$ be a separable Hilbert space. Suppose $x,y \in H^2(\Dd,E).$ Then 
 $$x\dot{\we} y \in H^1(\Dd, \we^2 E).$$
\end{proposition}

\begin{proof}
By Proposition \ref{wejanalytic}, $ x\dot{\we} y$ is analytic on $\Dd.$ By Proposition \ref{we}, for $0<r<1$ and $ 0\leq \theta \leq 2\pi,$ $$ \| (x \telwe y) (re^{i\theta})\|_{\we^2 E} \leq \|x(re^{i\theta})\|_E \|y(re^{i\theta})\|_E. $$

\noindent By Proposition \ref{we} and by Definition \ref{a.12}, 
 $$\begin{array}{cllll}
    \|x \telwe y \|_{H^1(\Dd,\we^2E)} &=\displaystyle \sup\limits_{0<r<1} \left(\frac{1}{2\pi}\int_0^{2\pi} \|(x \telwe y)(re^{i\theta})\|_{\we^2E}\; d\theta\right)\vspace{2ex} \\
    &\leq \displaystyle \sup\limits_{0<r<1} \left(\frac{1}{2\pi}\int_0^{2\pi} \|x(re^{i\theta})\|_{E} \|y(re^{i\theta})\|_{E}\; d\theta\right),
   \end{array}$$ for $0<r<1$ and $ 0\leq \theta \leq 2\pi.$
Also, by H\"{o}lder's inequality,  for $0<r<1$ and $ 0\leq \theta \leq 2\pi,$
$$ \frac{1}{2\pi}\int\limits_0^{2\pi} \|x(re^{i\theta})\|_E \|y(re^{i\theta})\|_E \;d\theta \leq  \left(\frac{1}{2\pi}\int\limits_0^{2\pi} \|x(re^{i\theta})\|_E^2 \; d\theta \right)^{1/2} 
\left(\frac{1}{2\pi}\int\limits_0^{2\pi} \|y(re^{i\theta})\|_E^2 \; d\theta \right)^{1/2},  $$hence
$$  \|x \telwe y \|_{H^1(\Dd,\we^2E)} \leq \|x\|_{H^2(\Dd,E)} \|y\|_{H^2(\Dd,E)}.$$

\noindent Consequently, $x\telwe y \in H^1 (\Dd, \we^2 E).$
\end{proof}

 \begin{proposition}\label{xweyh2}
 Let $E$ be a separable Hilbert space, let $x\in H^2(\Dd,E)$ and let $y \in H^{\infty}(\Dd, E).$ Then
  $$x \dot{\wedge} y \in H^2 (\Dd, \wedge^2 E).$$
 \end{proposition}
 \begin{proof}
 
 By Proposition \ref{wejanalytic}, $ x\dot{\we} y$ is analytic on $\Dd.$ By Proposition \ref{a.18}, for $$0<r<1, \quad 0\leq \theta \leq 2\pi,$$we have 
 $$\| (x \telwe y)(re^{i\theta})\|_{\we^2 E}\leq \|x(re^{i\theta})\|_E \|y(re^{i\theta})\|_E.$$
  Thus,
 \[ \begin{array}{cllllllll}
  \| x \telwe y\|_{H^2(\Dd,\we^2 E)}&=  \displaystyle \sup\limits_{0<r<1} \left(\frac{1}{2\pi}\int_0^{2\pi} \|(x\telwe y)(re^{i\theta})\|_{\we^2 E}^2 \;d\theta\right)^{1/2}\\
  &\leq \displaystyle \sup\limits_{0<r<1} \left(\frac{1}{2\pi}\int_0^{2\pi} \|x(re^{i\theta})\|_{E}^2 \|y(re^{i\theta})\|_{E}^2 d\theta\right)^{1/2} \\
&\leq \|y\|_\infty   \sup\limits_{0<r<1} \left(\displaystyle\frac{1}{2\pi}\int_0^{2\pi} \|x(re^{i\theta})\|_{E}^2 d\theta \right)^{1/2} < \infty .\end{array}
\]\end{proof}

\begin{definition} Let $E$ be a Hilbert space.
We say that a family of $\{ f_\lambda \}_{ \lambda \in \Lambda}$
of maps from $\Tt$ to $E$ is {\em pointwise orthonormal} on $\Tt$, if for all $z$ in a set of full measure in $\Tt$, 
the set of vectors $\{ f_\lambda(z) \}_{ \lambda \in \Lambda}$ is 
orthonormal in $E$.
\end{definition}

\begin{definition}\label{POC-def} Let $E$ be a separable Hilbert space.
Let $F$ be a subspace of $L^2(\Tt, E)$ and let $X$ be a subset of $L^2(\Tt,E).$
We define the \emph{pointwise orthogonal complement} of $X$ in $F$ to be the set
$$\Poc(X,F) = \{ f \in F: f(z)\perp \{x(z):x \in X\}\; \text{for almost all}\;z \in \Tt\}.$$
\end{definition}

 \begin{proposition}\label{wel2conv}  Let $E$ be a separable Hilbert space, and let 
$\; \xi_0,\xi_1, \cdots, \xi_j \in L^\infty(\Tt,E) $ be a pointwise orthonormal set on $\Tt$, and let $x\in L^2(\Tt,E)$.  Then
  $$ \xi_0 \dot{\we} \xi_1 \dot{\we} \cdots \dot{\we} \xi_j \dot{\we} x \in L^2 (\Tt, \wedge^{j+2} E),$$
and 
\[
\| \xi_0 \dot{\we} \xi_1 \dot{\we} \cdots \dot{\we} \xi_j \dot{\we} x \|_{L^2 (\Tt, \wedge^{j+2} E)} \le
\| x \|_{L^2(\Tt,E)}.
\]
Furthemore,
\[
\| \xi_0 \dot{\we} \xi_1 \dot{\we} \cdots \dot{\we} \xi_j \dot{\we} x \|_{L^2 (\Tt, \wedge^{j+2} E)} =
\| x \|_{L^2(\Tt,E)}
\]
if and only if $ x \in \Poc(\{\xi_0,\xi_1, \cdots, \xi_j\}, L^2(\Tt,E))$.
 \end{proposition}
 \begin{proof}
By Lemma \ref{weon}, for almost all $z \in \Tt$, 
\[
\|\xi_0(z) \we \xi_1(z) \we \cdots \we \xi_j(z) \we x(z) \|^2_{\we^{j+2}E} =
\| x(z)\|^2_{E} - \displaystyle\sum\limits_{i=0}^j 
|\langle x(z), \xi_i(z) \rangle_{E}|^2 \le \| x(z) \|^2_{E}.
\]
 Thus,
\[ \begin{array}{cllllllll}
  &\| \xi_0 \dot{\we} \xi_1 \dot{\we} \cdots \dot{\we} \xi_j \dot{\we} x\|_{L^2(\Tt,\wedge^{j+2} E)}\vspace{2ex}\\
&=  \displaystyle \left(\frac{1}{2\pi}\int_0^{2\pi} 
\|\xi_0(e^{i\theta}) \we \xi_1(e^{i\theta}) \we \cdots \we \xi_j(e^{i\theta})\we x
(e^{i\theta})\|_{\we^{j+2} E}^2 \;d\theta\right)^{1/2}\\
  &\leq \displaystyle \left(\frac{1}{2\pi}\int_0^{2\pi} \|x(e^{i\theta})\|_{E}^2 d\theta\right)^{1/2} = \| x \|_{L^2(\Tt,E)} < \infty. 
\end{array}
\]
\end{proof}

\section{Pointwise creation operators, orthogonal complements and linear spans}\label{orthog_complement}

\begin{definition}\label{pwcre} Let $E$ be a separable Hilbert space. 
	Let $\xi \in H^\infty(\Dd, E)$. We define the \emph{pointwise creation operator} $$C_\xi \colon H^2(\Dd,E) \to H^2(\Dd, \we^2E)$$ by 
	$$ C_\xi f = \xi \telwe f, \; \text{for} \; f \in H^2(\Dd,E).$$ 
\end{definition}

\begin{remark}\label{genfatouwe}  Let $E$ be a separable Hilbert space. Let $\xi \in H^\infty(\Dd, E)$ and let $f \in H^2(\Dd, E)$.
 By the generalized Fatou's Theorem \cite[Chapter V]{NagyFoias}, the radial limits 
	$$\lim_{r\to 1}\xi(r\eiu)\underset{\|\cdot\|_E}{=} \tilde{\xi}(\eiu), \quad \lim_{r\to 1} f(r\eiu)\underset{\|\cdot\|_E}{=} \tilde{f}(\eiu) \;\;
(0<r<1)
$$ exist almost everywhere on $\Tt$ and define functions $\tilde{\xi} \in L^\infty(\Tt,E)$ and $\tilde{f}\in L^2(\Tt,E)$ respectively, which satisfy the relations 
	$$ \lim_{r \to 1} \| \xi(r\eiu) - \tilde{\xi}(\eiu) \|_E=0,\quad \lim_{r \to 1} \| f(r\eiu) - \tilde{f}(\eiu) \|_{E}=0 \;\;
(0<r<1) $$ 
for almost all $ \eiu  \in \Tt$.
\end{remark}\index{$\tilde{f}$}

\begin{lemma}\label{xitelweh21} Let $E$ be a separable Hilbert space.
 Let $\xi\in H^\infty(\Dd, E)$ and let $f \in H^2(\Dd, E).$ Then the radial limits $\lim_{r \to 1} (\xi(r\eiu) \we f(r\eiu))$ exist for almost all $\eiu \in \Tt$ and define functions in $L^2(\Tt,\we^2 E).$
\end{lemma}
\begin{proof}
By Proposition \ref{weopiscontinuous1}, the bilinear operator $\Lambda\colon E\times E \to \we^2E$ is a continuous operator for the norms of $E$ and $\we^2 E.$ By Remark \ref{genfatouwe}, the functions $\xi \in H^\infty(\Dd,E)$ and \newline $f\in H^2(\Dd,E)$ have radial limit functions $\tilde{\xi} \in L^\infty(\Tt,E)$ and $\tilde{f} \in L^2(\Tt,E)$. Also, by Proposition \ref{xweyh2}, $\xi \telwe f \in H^2(\Dd,\we^2E).$ Hence
$$\lim_{r \to 1} \| \xi(r\eiu)\we f(r\eiu) - \tilde{\xi}(\eiu) \we \tilde{f}(\eiu) \|_{\we^2E}=0  \quad \text{ almost everywhere on}\quad \Tt$$ and we conclude that

$$ \lim_{r\to 1} (\xi(r\eiu) \we f(r\eiu) ) \underset{\|\cdot\|_{\we^2E}}{=} \tilde{\xi}(\eiu) \we \tilde{f}(\eiu)\; \text{almost everywhere on}\; \Tt.$$

\noindent This shows that the radial limits 
$$ \lim_{r\to 1} (\xi(r\eiu) \we f(r\eiu) ) $$ exist almost everywhere on $\Tt$ and, by Lemma \ref{a.18}, define functions in $L^2(\Tt,\we^2E).$
 Hence one can consider $(C_\xi f)(z)= (\xi \telwe f)(z)$ to be defined for either all $z \in \Dd$ or for almost all $z \in \Tt.$
\end{proof}

\begin{remark}\label{H2subsetL2}  Let $E$ be a separable Hilbert space.
By \cite[Chapter 5, Section 1]{NagyFoias}, for any separable Hilbert space $E$, the map 
$f \mapsto \tilde{f}$ is an isometric embedding of $H^2(\Dd,E)$ in $L^2(\Tt,E)$, where 
$ \tilde{f}(\eiu)= \lim_{r \to 1} f(r\eiu).$ Since $H^2(\Dd,E)$ is complete and the embedding is isometric, the image of the embedding is complete, and therefore is closed in $L^2(\Tt, E).$ Therefore, the space $H^2(\Dd,E)$ is identified isometrically with a closed linear subspace of $L^2(\Tt, E).$ 
In future we shall use the same notation for $f$ and $\tilde{f}.$
\end{remark}

Our next aim is to show that $ \Poc(X,F)$ is a closed subspace of $F.$ We are going to need the following results.
 \begin{lemma}\label{fiscontl2}
 Let $E$ be a Hilbert space and let $x \in L^2(\Tt,E).$ The function \linebreak$\phi \colon L^2(\T, E) \to \mathbb{C}$ given by 
	$$
\phi (g) = \displaystyle\frac{1}{2\pi}\int\limits_0^{2\pi} |\langle g(e^{i\theta}) , x (e^{i\theta}) \rangle_{E}| \; d\theta
$$  
is continuous.
\end{lemma}
\begin{proof}
Consider $g_0 \in L^2(\Tt,E).$	For any $\epsilon>0,$ we are looking for a $\delta>0$ such that
 $$\| g -g_0\|_{L^2(\Tt,E)}=\left(\displaystyle\frac{1}{2\pi}\int\limits_0^{2\pi}
	\| g(e^{i\theta}) - g_0 (e^{i\theta})\|_{E}^2 \; d\theta\right)^{1/2}<\delta$$implies
	$$ |\phi (g) - \phi(g_0) | <\epsilon.$$
	
	\noindent Note that

$$\begin{array}{cllllll}	
	|\phi (g) - \phi(g_0) | &= \left| \displaystyle\frac{1}{2\pi}\int\limits_0^{2\pi} |\langle g(e^{i\theta}) , x (e^{i\theta}) \rangle_{E}| d\theta - \displaystyle\frac{1}{2\pi}\int\limits_0^{2\pi} |\langle g_0 (e^{i\theta}) , x (e^{i\theta}) \rangle_{E}| d\theta \right| \vspace{2ex}  \\
	&= \left| \displaystyle\frac{1}{2\pi}\int\limits_0^{2\pi} \left(|\langle g(e^{i\theta}) , x (e^{i\theta}) \rangle_{E}|-|\langle g_0 (e^{i\theta}) , x (e^{i\theta}) \rangle_{E}|\right) d\theta \right|.\end{array}$$
	For each $\eiu \in \Tt,$ by the reverse triangle inequality, the integrand satisfies 
	$$ \begin{array}{cllll} |\langle g(e^{i\theta}) , x (e^{i\theta}) \rangle_{E}|-|\langle g_0 (e^{i\theta}) , x (e^{i\theta}) \rangle_{E}|&\leq \left|\langle g(e^{i\theta}) , x (e^{i\theta}) \rangle_{E} - \langle g_0 (e^{i\theta}) , x (e^{i\theta}) \rangle_{E} \right|\vspace{2ex} \\
	&=| \langle(g\eiu)- g_0(\eiu), x(\eiu)\rangle_E|,\end{array}$$

	\begin{equation}\label{phileq}|\phi (g) - \phi(g_0) |\leq \displaystyle\frac{1}{2\pi}\int\limits_0^{2\pi} |\langle g(e^{i\theta})-  g_0 (e^{i\theta}), x (e^{i\theta}) \rangle_{E}| d\theta.\end{equation}
	
	\noindent By the Cauchy-Schwarz inequality,
	\begin{equation}\label{phil2<d}\begin{array}{cll} \begin{array}{cllll} &\displaystyle\frac{1}{2\pi}\int\limits_0^{2\pi} |\langle g(e^{i\theta})-  g_0 (e^{i\theta}), x (e^{i\theta}) \rangle_{E}|\; d\theta\vspace{2ex} \\ &\leq \left(
	\displaystyle\frac{1}{2\pi}\int\limits_0^{2\pi} \| g(e^{i\theta})-  g_0 (e^{i\theta})\|_{E}^2 \; d\theta \right)^{1/2} \left(
	\displaystyle\frac{1}{2\pi}\int\limits_0^{2\pi} \| x(e^{i\theta})\|_{E}^2 \; d\theta \right)^{1/2}. \end{array}\end{array}\end{equation}
For the given $\epsilon >0,$ let $\delta$ be equal to $\displaystyle\frac{\epsilon}{\|x\|_{L^2(\Tt,E)} +1},$ and let 
$$\left(\displaystyle\frac{1}{2\pi}\int\limits_0^{2\pi} \| g(e^{i\theta})-  g_0 (e^{i\theta})\|_{E}^2 \; d\theta \right)^{1/2} < \delta.$$ 
By equations  (\ref{phileq}) and (\ref{phil2<d}), 
$$\begin{array}{clll}|\phi (g) - \phi(g_0) |&\leq \left(	\displaystyle\frac{1}{2\pi}\int\limits_0^{2\pi} \| g(e^{i\theta})-  g_0 (e^{i\theta})\|_{E}^2 \; d\theta \right)^{1/2} \|x\|_{L^2(\Tt,E)} \vspace{2ex} \\ &< \displaystyle \frac{\epsilon}{\|x\|_{L^2(\Tt,E)}+1} \|x\|_{L^2(\Tt,E)} < \epsilon.\end{array}$$ Hence $\phi$  is a continuous function. \end{proof}

\begin{proposition}\label{vclosed} Let $E$ be a separable Hilbert space.
Let $\eta \in L^2(\Dd,E)$. Then
\begin{enumerate}
	\item[(i)] The space $V= \{ f \in H^2 (\Dd,E):  \langle f(z) , \eta(z) \rangle_{E} =0 \; \text{for almost all}\; z \in \Tt \}$ is a closed subspace of $H^2(\Dd, E).$ 
\item[(ii)] The space  $V=\{ f \in L^2(\Tt,E): \langle f(z), \eta(z) \rangle_{E} =0 \;\text{for almost all}\; z\in \Tt  \} $ is a closed subspace of $L^2(\Tt,E).$ 
\end{enumerate}
\end{proposition}
\begin{proof}
\noindent\text{(i).}		$V$ is a linear subspace of $H^2(\Dd, E)$ since for $\lambda, \mu \in \Cc,$ $\psi , k \in V$ and for almost all $z\in \Tt,$ 
	
	$$\langle \lambda \psi(z) + \mu k(z) , \eta (z) \rangle_{E} = \lambda \langle \psi(z), \eta (z) \rangle_{E} + \mu \langle k(z) , \eta (z) \rangle_{E} =0,  $$
	hence  $\lambda \psi + \mu k \in V.$
	
	\noindent Now suppose that the sequence of functions $(g_n)_{n=1}^\infty$ in $V$ converges to a function $g.$ We need to show that $g \in V.$ 
	Since $g_n \in V$ for all $n\in \mathbb{N},$ we have 
\begin{equation}\label{gn=0}
\langle g_n(z), \eta(z)\rangle_{E}=0 \;\text{for almost all}\; z\in \Tt. 
\end{equation}
	
	\noindent Consider the function 
	$\phi \colon H^2(\Dd,E)\to \Cc$ given by 
	$$ \phi(f) =  \displaystyle\frac{1}{2\pi} \int\limits_0^{2\pi} |\langle f (e^{i\theta}) , \eta (e^{i\theta}) \rangle_{E} |\; d\theta.$$
		\noindent Then, by equation (\ref{gn=0}), we have  $$\phi(g_n)= \displaystyle\frac{1}{2\pi} \int\limits_0^{2\pi} |\langle g_n (e^{i\theta}) , \eta (e^{i\theta}) \rangle_{E} |\; d\theta=0.$$
By Remark \ref{H2subsetL2} and Lemma \ref{fiscontl2}, $\phi$ is a continuous function on $H^2(\Dd,E),$ thus 
$$\lim\limits_{n \to \infty} \phi(g_n) = \phi(g),$$
and so 
$$ \displaystyle\frac{1}{2\pi} \int\limits_0^{2\pi} |\langle g (e^{i\theta}) , \eta (e^{i\theta}) \rangle_{E} |\; d\theta
= \lim\limits_{n \to \infty}  \displaystyle\frac{1}{2\pi} \int\limits_0^{2\pi} |\langle g_n (e^{i\theta}) , \eta (e^{i\theta}) \rangle_{E} | \; d\theta =0.$$ 
	Thus $|\langle g (e^{i\theta}) , \eta (e^{i\theta}) \rangle_{E} | =0$ for almost all $\eiu \in \Tt,$ and, hence, $g \in V.$ 
	We have proved that $V$ is a closed subspace of $H^2(\Dd, E).$  

\noindent{\text{(ii).}} The proof is similar to \text{(i)}. \end{proof}	

Our motivation for the next theorem is the following. 
In \cite{YCL2020} spaces of the form 
$$ \xi_0 \telwe \dots \telwe \xi_j \telwe H^2(\Dd,E)$$
play a crucial role in the superoptimal Nehari problem. They are the domains of Hankel-type operators whose norms are ``superoptimal singular values" of error functions $G-Q$, where $G$ is a given continuous approximand and $Q$ is its analytic approximation.

\begin{theorem}\label{xjclosed}
Let $E$ be a separable Hilbert space, let  $\xi_0, \xi_1, \cdots, \xi_j \in H^\infty(\Dd, E)$. Suppose that the set $\{\xi_i(z)\}_{i=0}^{j}$ is  orthonormal in $E$ for almost every $z \in \Tt.$ Then 
 $$ \xi_0 \telwe \dots \telwe \xi_j \telwe H^2(\Dd,E)$$
		is a closed subspace of $H^2(\Dd,\we^{j+2}E).$
	\end{theorem}
\begin{proof}
By Proposition \ref{wejanalytic}, for every $x \in H^2(\Dd,E),$ 
$$ \xi_0 \telwe \xi_1 \telwe\cdots \telwe \xi_j \telwe x$$ is analytic on $\Dd.$ By Proposition \ref{wel2conv}, since $\xi_0,\xi_1,\dots,\xi_j$ are pointwise orthogonal on $\Tt,$
	$$\| \xi_0 \telwe \xi_1 \telwe \cdots \xi_j \telwe x\|_{L^2(\Tt,\we^{j+2}E)} < \infty. $$ Thus, for every $x \in H^2(\Dd,E),$
	$$\xi_0 \telwe \xi_1 \telwe\cdots \telwe \xi_j \telwe x \in H^2(\Dd,\we^{j+2}E). $$	

 Let us first show that $\xi_0 \telwe H^2(\Dd,E)$ is a closed subspace of $H^2(\Dd,\we^2E).$ Observe that, by Proposition \ref{xitelweh21}, $\xi_0 \telwe H^2(\Dd,E)\subset H^2(\Dd,\we^2E).$ 
Let 
$$\Xi_0 =\{ f \in H^2(\Dd,E): \langle f(z), \xi_0 (z) \rangle_{E}=0\quad \text{almost everywhere on} \; \Tt \}.$$ 
Consider a vector-valued function $w\in H^2(\Dd,E).$ 
For almost every  $z\in \T,$ we may write $w$ as 
	$$w(z)= w(z) -\langle w(z),\xi_0(z)\rangle_{E}\xi_0(z)+\langle w(z),\xi_0(z)\rangle_{E}\xi_0(z).$$ 
Then, for all $w\in H^2(\Dd,E)$ and for  almost every $z\in \T,$	
	$$\begin{array}{clllll}(\xi_0 \telwe w)(z)& = \xi_0(z)\we \big(w(z) -\langle w(z),\xi_0(z)\rangle_{E}\xi_0(z)+\langle w(z),\xi_0(z)\rangle_{E}\xi_0(z)\big)\vspace{2ex}\\&=\xi_0(z)\we \big(w(z) -\langle w(z),\xi_0(z)\rangle_{E}\xi_0(z) \big)  \end{array}$$
due to the pointwise linear dependence of $\xi_0$ and 
$z \mapsto \langle w(z),\xi_0(z) \rangle_E \xi_0(z) $ almost everywhere on $\T.$ 
Note that 
$$w(z)- \langle w(z),\xi_0(z)\rangle_{E}\xi_0(z)\in \Xi_0,$$ 
thus 
	$$\xi_0 \telwe H^2(\Dd,E) \subset \xi_0 \telwe \Xi_0. $$ 
By Proposition \ref{vclosed}, $\Xi_0$ is a closed subspace of $H^2(\Dd,E),$ hence
	$$ \xi_0 \telwe H^2(\Dd,E) \supset \xi_0 \telwe \Xi_0,$$
and so,
	$$\xi_0 \telwe H^2(\Dd,E) = \xi_0 \telwe \Xi_0 .$$
Consider the mapping 
	$$C_{\xi_0}\colon \Xi_0 \to \xi_0 \telwe \Xi_0 $$
given by
	$$C_{\xi_0} w = \xi_0 \telwe w $$
for all $w\in \Xi_0.$  Notice that, by  assumption,
$\|\xi_0(e^{i\theta})\|_{E}^2=1$ for almost every $\eiu \in  \Tt.$ Therefore, for any $w \in \Xi_0,$ we have 
	$$\begin{array}{cllllll}
	\|\xi_0 \telwe w\|_{L^2(\Tt,\we^2E)}^2 &=\displaystyle \frac{1}{2\pi} \int\limits_0^{2\pi} \langle \xi_0 \telwe w, \xi_0 \telwe w \rangle (e^{i\theta})d\theta \vspace{3ex} \\
	&=  \displaystyle \frac{1}{2\pi} \int\limits_0^{2\pi}\left( \|\xi_0(e^{i\theta})\|_{E}^2 \|w(e^{i\theta})\|_{E}^2 - |\langle w(e^{i\theta}), \xi_0(e^{i\theta})\rangle|^2\right)\; d\theta \vspace{3ex}\\
	&= \|w\|_{L^2(\Tt, E)}^2, 
	\end{array}$$
since
	$w$ is pointwise orthogonal to $\xi_0$ almost everywhere on $\Tt.$  Thus the mapping 
$$C_{\xi_0}\colon  \Xi_0 \to \xi_0 \telwe \Xi_0$$ 
is an isometry. Furthermore, 
$C_{\xi_0}\colon  \Xi_0 \to \xi_0 \telwe \Xi_0$ 
is a surjective mapping, 
	thus $\Xi_0$ and $\xi_0 \telwe \Xi_0$ are isometrically isomorphic. Therefore, since $\Xi_0$ is a closed subspace of $H^2 (\Dd, E),$ the space $\xi_0 \telwe \Xi_0$ is a closed subspace of $H^2(\Dd,\we^2E)$. Hence $\xi_0 \telwe H^2(\Dd,E)$ is a closed subspace of $H^2(\Dd,\we^2E).$
	
	To prove that 
	$\xi_0 \telwe \dots \telwe \xi_j \telwe H^2(\Dd,E) $ is a closed subspace of $H^2(\Dd,\we^{j+2}E)$, let us consider
	$$\Xi_j =\{f \in H^2(\Dd,E):\langle f(z), \xi_i (z)\rangle_{E} =0, \; \text{for}\; i=0,\cdots,j \} $$to be the pointwise orthogonal complement of $\xi_0, \dots,\xi_j$ in $H^2(\Dd,E).$ Let $\psi \in H^2(\Dd,E).$ We may write $\psi$ as 
	$$\psi(z) = \psi(z) - \sum\limits_{i=0}^j \langle \psi(z), \xi_i(z)\rangle_{E}\xi_i(z) +\sum\limits_{i=0}^j \langle \psi(z), \xi_i(z)\rangle_{E}\xi_i(z).  $$Then, for all $\psi \in H^2(\Dd,E)$ and for almost all $z \in \Tt,$ 
	
	$$(\xi_0 \telwe \cdots \telwe \xi_j \telwe \psi)(z) = \xi_0(z)\we\cdots\we \left( \psi(z)- \sum\limits_{i=0}^j \langle \psi(z), \xi_i(z)\rangle_{E}\xi_i(z)  \right)$$
due to the pointwise linear dependence of $\xi_k$ and 
$z \mapsto \langle \psi, \xi_k\rangle_E \xi_k$ almost everywhere on $\T.$
	
	\noindent Notice that  $\left( \psi(z)- \sum\limits_{i=0}^j \langle \psi(z), \xi_i(z)\rangle_{E}\xi_i(z)  \right)$ is in $\Xi_j,$ thus
	$$\xi_0 \telwe \cdots \telwe \xi_j \telwe H^2(\Dd,E) \subset \xi_0 \telwe \cdots \telwe \xi_j \telwe \Xi_j. $$
The reverse inclusion holds by the definition of $\Xi_j,$ hence 
	
	$$\xi_0 \telwe \cdots \telwe \xi_j \telwe H^2(\Dd,E) = \xi_0 \telwe \cdots \telwe \xi_j \telwe \Xi_j.$$
	
	Consequently, in order to prove the proposition it suffices to show that $\xi_0 \telwe \cdots \telwe \xi_j \telwe \Xi_j $ is a closed subspace of $H^2(\Dd, \we^{j+2} E).$
	By Proposition \ref{vclosed}, $\Xi_j$ is a closed subspace of $H^2(\Dd, E),$ being a finite intersection of closed subspaces. 

By Lemma \ref{weon}, for any $f \in \Xi_j,$ we get  	

$$ \begin{array}{clllll}
	\|\xi_0 \telwe \xi_1 \telwe\cdots\telwe \xi_j \telwe f\|_{L^2 ( \Tt, \we^{j+2} E)}^2 \vspace{2ex} &=\displaystyle\frac{1}{2\pi} 
	\int_0^{2\pi} \|\xi_0(e^{i\theta}) \we \xi_1(e^{i\theta}) \we\cdots\we \xi_j(e^{i\theta}) \we f(e^{i\theta})\|_{\we^{j+2} E}^2 d\theta \\
&=\displaystyle\frac{1}{2\pi} 
	\int_0^{2\pi} \|f(e^{i\theta}) - \displaystyle\sum\limits_{i=1}^j \langle f(e^{i\theta}) , \xi_i(e^{i\theta}) \rangle \xi_i(e^{i\theta})\|_{E}^2 d\theta \vspace{2ex}\\
&= \displaystyle\frac{1}{2\pi} \int_0^{2\pi} \|f(e^{i\theta})\|_{E}^2 d\theta = \|f\|_{L^2 ( \Tt, E)}^2.\end{array}$$

	Thus $$(\xi_0 \telwe \xi_1 \telwe\cdots\telwe \xi_j \telwe \cdot) \colon \Xi_j \to  \xi_0 \telwe \xi_1 \telwe\cdots\telwe \xi_j \telwe \Xi_j $$ 
	is an isometry. 
	Furthermore  
$$(\xi_0 \telwe \xi_1 \telwe\cdots\telwe \xi_j \telwe \cdot)\colon \Xi_j \to \xi_0 \telwe \xi_1 \telwe\cdots\telwe \xi_j \telwe \Xi_j $$ 
is a surjective mapping, thus $\Xi_j$ and $\xi_0 \telwe \cdots \telwe \xi_j \telwe \Xi_j$ are isometrically isomorphic. Therefore, since $\Xi_j$ is a closed subspace of $H^2(\Dd,E),$ the space $\xi_0 \telwe \cdots \telwe \xi_j \telwe \Xi_j $ is a closed subspace of $H^2(\Dd,\we^{j+2}E).$ Hence 
	$$\xi_0 \telwe \cdots \telwe \xi_j \telwe H^2(\Dd,E) $$ is a closed subspace of $H^2(\Dd,\we^{j+2}E).$
\end{proof}

\begin{proposition}\label{poclosed} Let $E$ be a separable Hilbert space,
	let $F$ be a subspace of $L^2(\Tt,E)$ and let $X$ be a subset of $L^2(\Tt,E).$	The space 
$$\Poc(X,F)= \{ f \in F \; : \; f(z) \perp \{ x (z) : x \in X\} \; \text{for almost all}\; z \in \Tt \}$$ is 
a closed subspace of $F.$
\end{proposition}
\begin{proof} It follows from Proposition \ref{vclosed}, since $\Poc(X,F)$ is
an intersection of closed subspaces 
$V_x = \{ f \in F: \langle f(z), x(z) \rangle_E =0 \;\; \text{for almost all } \;\; z \in \T \}$ over $x \in F$.
 \end{proof}	
\begin{remark}
	By the generalized Fatou's Theorem, for $E$ a separable Hilbert space, the space $H^\infty(\Dd,E)$ can be identified with a closed subspace of $L^\infty(\Tt,E).$
\end{remark}
\begin{definition}
	Let $E$ be a separable Hilbert space. Let $f \in H^p (\Dd,E),$ for $1 \leq p \leq \infty.$ 
	By the generalized Fatou's Theorem (see \cite{NagyFoias}, p. 186), the radial limit 	
	$$\lim_{r \to 1}f(r\eiu)\underset{\|\cdot\|_E}{=}\tilde{f}(\eiu)
\;\; (0<r<1)$$ 
exists almost everywhere on $\Tt$ and defines a function $\tilde{f} \in L^p(\Tt,E).$ The set of points on $\Tt$ at which the above limit does not exist, will be called the \emph{singular set} of the function $f$ and will be denoted by $N_f.$
\end{definition}\raggedbottom Note that the singular sets of functions in $H^p(\Dd,E)$ for $1 \leq p \leq \infty$ are null sets with respect to Lebesgue measure.

\begin{definition}\label{pls} Let $E$ be a separable Hilbert space.
Let $F$ be a subspace of $L^2(\Tt, E)$ and let $X$ be a subset of $L^2(\Tt,E).$ We define the \emph{pointwise linear span} of $X$ in $F$ to be the set 
$$\Pls(X,F)= \{f \in F: f(z)\in \spn\{x(z): x \in X\}\;\text{for almost all} \;z\in \Tt\}. $$ 
\end{definition}

Recall Definition \ref{pwcre}, for a separable Hilbert space $E$ and for $\xi \in H^\infty(\Dd,E)$, the pointwise creation operator  $C_\xi$ is defined by
$$C_\xi\colon H^2(\Dd,E)\to H^2(\Dd,\we^2E),\ f\mapsto \xi\telwe f,$$ 
where 
\[
(\xi \telwe f)(z)=\xi(z)\we f(z) \mbox{ for all }z \in \Dd.
\]

\begin{proposition}\label{kercxi} Let $E$ be a separable Hilbert space.
For $\xi \in H^\infty(\Dd, E),$ 
$$\ker C_\xi \subset  \Pls(\{\xi\}, H^2(\Dd,E)).$$
\end{proposition}
\begin{proof}
We have 
\[ \begin{array}{clll} \ker C_\xi &= \{f \in H^2(\Dd,E): (\xi \telwe f)(z)=0\; \text{for all}\; z \in \Dd \}\\
&= \{ f \in H^2(\Dd,E): \xi(z) \we f(z)=0 \; \text{for all}\; z \in \Dd  \} \\
&= \{f \in H^2(\Dd,E) :\xi(z), f(z) \;\text{are pointwise linearly dependent for all}\;z \in \Dd  \} \\
&\subset   \Pls(\{\xi\}, H^2(\Dd,E)).
\end{array} 
\]\end{proof}

\begin{example} {\rm
Let $E=\Cc^2.$ We can find functions $f,g \in H^2(\Dd,E)$ such that \linebreak$f\in \Poc(\{g\}, H^2(\Dd,E))$ but it is false that $\langle f(z), g(z)\rangle_E =0 $ for all $z \in \Dd.$ 

\noindent Choose 
$$g(z) = \begin{pmatrix}
z \\ z^2 
\end{pmatrix}, \quad f(z)= \begin{pmatrix}
f_1(z) \\ f_2(z)
\end{pmatrix} \quad \text{for}\quad z\in \Dd.$$ 

\noindent	Then $$f \in \Poc(\{g\},H^2(\Dd,E))$$ is equivalent to $$\langle {f}(z), {g}(z) \rangle_E =0 \quad \text{for almost all} \quad z\in \Tt.$$ The later is equivalent to 
$$ \left\langle \begin{pmatrix}  z \\ z^2 \end{pmatrix} , \begin{pmatrix}
{f_1}(z) \\ {f_2}(z)
\end{pmatrix}  \right\rangle_E=0  \quad \text{for almost all} \quad z\in \Tt,$$ which holds if and only if $$ \bar{z} {f_1}(z) + \bar{z}^2 {f_2}(z) = 0 \quad \text{for almost all} \quad z \in \Tt.$$

\noindent Equivalently $$ {f_1}(z) = - \bar{z} {f_2}(z) \quad \text{for almost all} \quad z\in \Tt, $$ which in turn is equivalent to $$ f(z) = \begin{pmatrix} 
{f_1}(z) \\ -z {f_1}(z).
\end{pmatrix} \quad \text{for almost all} \quad z \in \Tt. $$ 

\noindent Now, for $z\in \Dd,$ 
$$\begin{array}{clll} 
\langle f(z), g(z) \rangle_E &= \left\langle \begin{pmatrix}
f_1 (z) \\ -zf_1(z) 
\end{pmatrix}, \begin{pmatrix}
z \\ z^2 
\end{pmatrix} \right\rangle_E \vspace{2ex} \\ 
&= \bar{z} f_1(z) - \bar{z} |z|^2 f_1(z)\vspace{2ex} \\ 
&= \bar{z} ( 1- |z|^2)f_1 (z).
\end{array} $$  So if we take $f_1(z) =1, f(z) = \begin{pmatrix}
1 \\ -z 
\end{pmatrix} $ 
for  $z\in \Dd,$ then $f \in \Poc(\{g\},H^2(\Dd,E))$, but $\langle f(z), g(z) \rangle_E \neq 0$ for all $z \in \Dd\setminus\{0\}.$

Thus it is not true  in general that $\Poc(\{g\},H^2(\Dd,E)) \subset \{g\}^{\perp}$.
}
\end{example}

\begin{lemma}\label{pocker} Let $E$ be a separable Hilbert space.
For $\xi \in H^\infty(\Dd, E), $
$$\Poc(\{\xi\},H^2(\Dd,E)) \subset H^2(\Dd,E) \ominus \Pls(\{\xi\}, H^2(\Dd,E))  .$$
\end{lemma}

\begin{proof}
Let  $f \in \Poc(\{\xi\},H^2(\Dd,E)).$ This is equivalent to $f \in H^2(\Dd,E)$ and 
$${f}(z)\perp {\xi}(z) \;\text{for all}\; z \in \Tt \setminus (N_f \cup N_\xi) ,$$
 where $N_f, N_\xi$ are the singular sets for the functions $f, \xi$ respectively. This in turn is equivalent to $f \in H^2(\Dd,E)$ and
$$ \langle {f}(z),{\xi}(z)\rangle_E=0 \;\text{for all}\; z\in \Tt \setminus (N_f \cup N_\xi) .$$ 
The latter implies $f \in H^2(\Dd,E)$ and 
$$\langle {f}(z),{g}(z)\rangle_E=0\;\text{for almost all}\;\; z \in \Tt \; \text{for all}\; g \in \Pls(\{\xi\},H^2(\Dd,E)).  $$Thus 
\[f \in H^2(\Dd,E)\ominus \Pls(\{\xi\},H^2(\Dd,E)).\qedhere\]\end{proof}
\begin{lemma}
 Let $E$ and $F$ be separable Hilbert spaces, and let $G\in L^\infty(\Tt, \mathcal{B}(F,E))$.  For every $x\in  L^2(\Tt,E)$, the function $Gx$, defined almost everywhere on $\Tt$ by
\[
(Gx)(z) = G(z)(x(z)),
\]
belongs to $L^2(\Tt,E)$.
\end{lemma}
\begin{proof}
For almost all $z\in\Tt$,
\[
\|(Gx)(z)\|_E = \|G(z)x(z)\|_E \leq \|G\|_{L^\infty(\Tt, \mathcal{B}(F,E))} \|x(z)\|_F.
\]
Thus
\begin{align*}
\|Gx\|_{L^2(\Tt,E)}^2 &=\frac{1}{2\pi} \int_0^{2\pi} \|Gx(e^{i\theta})\|_E^2 \ d\theta\\
	&\leq \frac{1}{2\pi} \int_0^{2\pi} \|G\|_{L^\infty(\Tt, \mathcal{B}(F,E))}^2 \|x(e^{i\theta})\|_F^2 \ d\theta\\
    &\leq \|G\|_{L^\infty(\Tt, \mathcal{B}(F,E))}^2 \|x\|^2_{L^2(\Tt,F)} < \infty.
\end{align*} \end{proof}

\begin{definition}\label{Toeplitz} Let $E$ and $F$ be separable Hilbert spaces.
Let $P_+\colon  L^2(\Tt,E)\to H^2(\Dd,E)$ be the orthogonal projection operator. Corresponding to any $G\in L^\infty(\Tt, \mathcal{B}(F,E))$ we define the \emph{Toeplitz operator}  with symbol $G$ to be the operator 
$$ T_G\colon  H^2(\Dd,F)\to H^2(\Dd,E)$$
given by 
$$T_Gx = P_+(Gx) \quad \text{for any}\quad x\in H^2(\Dd,F).$$
\end{definition}
\begin{definition}[\cite{NagyFoias}, p. 190]
	For a separable Hilbert space $E,$ a function $\xi \in H^\infty(\Dd, E)$ will be called \emph{inner} if for almost every $z \in \Tt,$
	$$ \|\xi(z)\|_E=1.$$
\end{definition}
\begin{theorem}\label{cx*cx}
	Let $E$ be a separable Hilbert space, and let $\xi \in H^\infty(\Dd,E).$
Consider the adjoint operator $$C_\xi^*\colon  H^2(\Dd,\we^2E)\to  H^2(\Dd, E).$$
{\rm (i)}  For every $g \in H^2(\Dd,E)$ and $f \in H^\infty(\Dd,E)$,
	$$ C_\xi^* (f \telwe g) = P_+ \alpha_f, $$
where $\alpha_f \in L^2(\Tt,E)$ is defined by 
	$$\alpha_f(\eiu) = \langle f(\eiu), \xi(\eiu)\rangle_E g(\eiu) - \langle g(\eiu), \xi(\eiu)\rangle_E f(\eiu) \quad  $$  for all $\eiu \in \Tt \setminus (N_f\cup N_\xi \cup N_g),$ and 
	$ P_+ \colon  L^2(\Tt,E) \to H^2(\Dd,E)$ is the orthogonal projection.
	 Here $N_f, N_g, N_\xi$ are the singular sets of the functions $f, g$ and $\xi$ respectively. 	

{\rm (ii)} Let $\xi \in H^\infty(\Dd,E)$ be an inner function. Then, for any $g \in H^2(\Dd,E),$
$$C_\xi^* C_\xi g =  P_{+} \alpha_\xi, $$
where $\alpha_\xi= g - \xi \xi^* g$.
Moreover 
$$ C_\xi^* C_\xi g = g -T_{\xi\xi^*}g,$$
where $T_{\xi\xi^*}\colon H^2(\Dd,E)\to H^2(\Dd,E)$ is 
the  Toeplitz operator with symbol $\xi\xi^*$.
\end{theorem}
\begin{proof} (i) By Proposition \ref{xweyh2}, $f \telwe g \in H^2(\Dd,\we^2 E).$ Now, for all $f\in H^\infty(\Dd,E),$ and	
	all $g,h \in H^2(\Dd,E)$, we have 
	$$\begin{array}{cllll}
	\langle C_\xi^* (f\telwe g), h \rangle_{H^2(\Dd,E)} &= \langle f \telwe g, C_\xi h\rangle_{H^2(\Dd,\we^2E)}\vspace{2ex} \\
&= \langle f \telwe g, \xi \telwe h \rangle_{L^2(\Tt,\we^2E)} \vspace{2ex}\\
	&= \displaystyle\frac{1}{2\pi} \int_{0}^{2\pi} \langle f(\eiu) \we g(\eiu), \xi(\eiu)\we h(\eiu)\rangle_{\we^2E}\;d\theta,\end{array}$$which, by Proposition \ref{we}, is equal to $$ \displaystyle\frac{1}{2\pi} \int_{0}^{2\pi} \det \begin{pmatrix}
	\langle f(e^{i\theta}), \xi(e^{i\theta}) \rangle_E & \langle f(e^{i\theta}) , h(e^{i\theta}) \rangle_E \\
	\langle g(e^{i\theta}) , \xi(e^{i\theta}) \rangle_E & \langle g(e^{i\theta}) , h (e^{i\theta}) \rangle_E \end{pmatrix}\; d\theta .$$ 
The latter in turn is equal to
$$\begin{array}{lll} 
	& \displaystyle \frac{1}{2\pi} \int_{0}^{2\pi} \left(\langle f(e^{i\theta}), \xi(e^{i\theta}) \rangle_E \langle g(e^{i\theta}) , h (e^{i\theta}) \rangle_E 
-\langle f(e^{i\theta}) , h(e^{i\theta}) \rangle_E\langle g(e^{i\theta}) , \xi(e^{i\theta}) \rangle_E \right)\; d\theta \vspace{2ex} \\
	&= \displaystyle \frac{1}{2\pi} \int_{0}^{2\pi} \left\langle  \langle f(e^{i\theta}), \xi(e^{i\theta}) \rangle_E g(e^{i\theta})
- \langle g(e^{i\theta}) , \xi(e^{i\theta}) \rangle_E f(e^{i\theta}) , h(\eiu)\right\rangle_E \;d\theta, \end{array}$$
which equals
	$$\begin{array}{llll}
	  \displaystyle \frac{1}{2\pi} \int_{0}^{2\pi} \langle \alpha_f(\eiu), h(\eiu) \rangle_E\;d\theta
	&= \langle \alpha_f, h \rangle_{L^2(\Tt,E)} \vspace{2ex}\\
	&= \langle P_{+}(\alpha_f) , h \rangle_{H^2(\Dd, E)},
\end{array}$$
where $$\alpha_f(\eiu)= \langle f(e^{i\theta}), \xi(e^{i\theta}) \rangle_E g(e^{i\theta}) - \langle g(e^{i\theta}), \xi(e^{i\theta}) \rangle_E f(e^{i\theta})$$for all $\eiu \in \Tt \setminus (N_\xi \cup N_f\cup N_g).$ Hence $C_\xi^* (f\telwe g) = P_+\alpha_f$ as required.
	
Let us show (ii).  By Part (i), for all $g, h  \in H^2(\Dd,E)$,
$$\begin{array}{cllll}
\langle C_\xi^* C_\xi g, h \rangle_{H^2(\Dd,E)} &=\langle C_\xi^* (\xi \telwe g), h \rangle_{H^2(\Dd,E)} \\
&= \langle P_{+}(\alpha_\xi), h \rangle_{H^2(\Dd, E)},
\end{array}$$
where $$\alpha_\xi(\eiu)= \langle \xi(e^{i\theta}), \xi(e^{i\theta}) \rangle_E g(e^{i\theta}) - \langle g(e^{i\theta}), \xi(e^{i\theta}) \rangle_E \xi(e^{i\theta})$$
for all $\eiu \in \Tt \setminus (N_\xi \cup N_g\cup N_h).$

Since $\xi$ is inner, $\|\xi(\eiu)\|_E=1$ almost everywhere on $\Tt,$ and so,
 $C_\xi^* C_\xi g = P_+ \alpha_\xi,$ where $\alpha_\xi = g- \xi\xi^* g.$ Hence $$ C_\xi^* C_\xi g = P_+(g-\xi\xi^* g)=g - T_{\xi\xi^*}g,$$
where $T_{\xi\xi^*}g = P_+(\xi\xi^*g)$ is the  Toeplitz operator with symbol $\xi\xi^*$.
\end{proof}

\begin{example} {\rm
There exists an inner function $\xi \in H^\infty(\Dd,\Cc^2)$ such that, for some \newline $h\in H^2(\Dd,\Cc^2),$ $C_\xi^*C_\xi h $ is not in the pointwise orthogonal complement of $\xi$ in $E.$ 

	Let $\xi \in H^\infty(\Dd,\Cc^2)$ be an inner function and let $h \in H^2(\Dd,\Cc^2).$ Let $\xi,h$ be given by
$$ \xi(z) =\frac{1}{\sqrt{2}} \begin{pmatrix}
1 \\ z
\end{pmatrix}, \quad h(z)= \begin{pmatrix}
1 \\1 
\end{pmatrix}\; \text{for all}\; z \in \Dd.$$ By Theorem \ref{cx*cx},
$$ C_\xi^* C_\xi h = P_+ \alpha,$$ where, 
 for all $z \in \Tt,$
$$\alpha (z) =\displaystyle\begin{pmatrix}
1 \\1 
\end{pmatrix} -\left\langle\begin{pmatrix}
1 \\1 
\end{pmatrix} , \frac{1}{\sqrt{2}} \begin{pmatrix}
1 \\ z
\end{pmatrix}\right\rangle_{\Cc^2}  \frac{1}{\sqrt{2}} \begin{pmatrix}
1 \\ z
\end{pmatrix} . $$
Calculations yield, for all $z \in \Tt,$
$$ \begin{array}{clllll}	
\alpha(z)&=   \begin{pmatrix}
1 \\1 
\end{pmatrix} - \frac{1}{\sqrt{2}} (1+\bar{z})  \frac{1}{\sqrt{2}} \begin{pmatrix}
1 \\ z
\end{pmatrix} \vspace{2ex}\\
&=\displaystyle \begin{pmatrix}
1 \\1 
\end{pmatrix} - \frac{1}{2} \begin{pmatrix}
1 + \bar{z} \\ z(1+\bar{z}) 
\end{pmatrix}  \end{array}
\vspace{2ex}$$ 
$$\begin{array}{clll}
& = \displaystyle \frac{1}{2} \begin{pmatrix}
1  -\bar{z} \\ 1- z 
\end{pmatrix}. 
\end{array}$$ 
Thus $$ (P_+ \alpha)(z) = \displaystyle\frac{1}{2} \begin{pmatrix}
1 \\ 1 -z
\end{pmatrix}\quad \text{for all}\quad z \in \Tt.$$
The latter expression is not in the pointwise orthogonal complement of $\xi$ in $\Cc^2,$ since for all $z \in \Tt,$
$$ \begin{array}{clllll}
\displaystyle \left\langle \frac{1}{2} \begin{pmatrix}
1 \\ 1-  z 
\end{pmatrix} ,  \frac{1}{\sqrt{2}} \begin{pmatrix}
1 \\ z
\end{pmatrix} \right\rangle_{\Cc^2}  &= \displaystyle\frac{1}{2\sqrt{2}} \begin{pmatrix}
1 & \bar{z}
\end{pmatrix} \begin{pmatrix}
1 \\ 1-  z 
\end{pmatrix} \vspace{2ex}\\
&= \displaystyle\frac{1}{2\sqrt{2}} (1+\bar{z}-|z|^2) = \displaystyle\frac{1}{2\sqrt{2}} \bar{z} \neq 0.

\end{array} $$
}
\end{example}
\begin{lemma}\label{cxipartiso}
Let $E$ be a separable Hilbert space. For every inner function \linebreak$\xi \in H^\infty(\Dd, E),$ 
$$\{ x \in H^2(\Dd,E): \|C_\xi x\|_{H^2(\Dd,\we^2E)} = \|x\|_{H^2(\Dd,E)}\} =\Poc(\{\xi\}, H^2(\Dd,E)).  $$ 
\end{lemma}

\begin{proof}
	By Remark \ref{H2subsetL2}, for every $x\in H^2(\Dd,E),$ $\|x\|_{H^2(\Dd,E)}=\|x\|_{L^2(\Tt,E)}.$ Hence 
$$\{x\in H^2(\Dd,E):\|C_\xi x\|_{H^2(\Dd,\we^2E)}^2=\|x\|_{H^2(\Dd,E)}^2 \}= \{ x \in H^2(\Dd,E): \|C_\xi x\|_{L^2(\Tt,\we^2E)}^2=\|x\|_{L^2(\Tt,E)}^2\}. $$

	By Proposition \ref{we}, the latter set is equal to 

$$\begin{array}{llll}
		& \left\{ x \in H^2(\Dd,E): \displaystyle \frac{1}{2\pi}\int_0^{2\pi} \det \begin{pmatrix}
		\langle \xi(\eiu) , \xi(\eiu) \rangle_E & \langle \xi(\eiu), x(\eiu) \rangle_E\\
		\langle x(\eiu) , \xi(\eiu) \rangle_E & \langle x(\eiu), x(\eiu) \rangle_E 
		\end{pmatrix}\;d\theta  \right.\vspace{2ex}\\ &\hspace{20ex}\left.=  \displaystyle\frac{1}{2\pi}  \int_0^{2\pi} \|x(\eiu)\|_E^2\;d\theta  \right\}.\end{array}$$
Since $\xi$ is inner, $\| \xi(\eiu)\|_E=1$ almost everywhere on $\Tt,$ hence the last set is equal to 		
		\[\begin{array}{lll}  &\left\{ x \in H^2(\Dd,E):  \displaystyle\frac{1}{2\pi}\int_0^{2\pi} (\|x(\eiu)\|_E^2- | \langle \xi(\eiu), x (\eiu) \rangle_E|^2)\;d\theta  \right. \vspace{2ex} \\ & \left. \hspace{20ex}=\displaystyle \frac{1}{2\pi}\int_0^{2\pi} \|x(\eiu)\|_E^2 \; d\theta  \right\} \vspace{2ex} \\
		&=  \left\{ x \in H^2(\Dd,E): \displaystyle \frac{1}{2\pi}\int_0^{2\pi} | \langle \xi(\eiu), x (\eiu) \rangle_E|^2 \;d\theta    = 0 \right\}\vspace{2ex}\\
		&=  \{ x \in H^2(\Dd,E) : {\xi}(\eiu)\perp {x}(\eiu) \;  \text{almost everywhere on}\; \Tt\} \vspace{2ex}\\
		&=\Poc(\{\xi\}, H^2(\Dd,E)).\end{array} 
\]\end{proof}

The following example illustrates the fact that, although  $C_\xi$ is an isometry on \linebreak$\Poc\{\xi,H^2(\Dd,\Cc^2)\},$ it is not a partial isometry on $ H^2(\Dd,\Cc^2)$.
\begin{example}
{\rm $C_\xi^* C_\xi$ fails to be a projection for some inner function 
$\xi \in H^\infty(\Dd,\Cc^2).$
Let us calculate $C_\xi^*C_\xi$ for 
$\xi(z)= \frac{1}{\sqrt{2}} \begin{pmatrix}
1 \\ z
\end{pmatrix},\; z \in \Dd$. 
By Theorem \ref{cx*cx}, for $h \in H^2(\Dd,\Cc^2)$
$$ C_\xi^* C_\xi h = P_+\alpha, $$ 
where, for all $z \in \Tt,$ $\alpha(z)$ is given by
$$
\begin{array}{clll}
\alpha(z) &= h(z) - \langle h(z) , \xi(z)\rangle_E \xi(z) \vspace{2ex}\\

&= \begin{pmatrix}
h_1(z) \\ h_2(z) 
\end{pmatrix}- \frac{1}{\sqrt{2}} (h_1(z) + \bar{z}h_2(z)) \frac{1}{\sqrt{2}} \begin{pmatrix}
1 \\ z
\end{pmatrix}\vspace{2ex}\\\end{array}$$
$$
\begin{array}{ll}
&= \begin{pmatrix}
h_1(z) - \frac{1}{2}(h_1(z)+ \bar{z}h_2(z))\\
h_2(z)- \frac{1}{2} (zh_1(z)+h_2(z)
\end{pmatrix}\vspace{2ex} \\

&= \frac{1}{2}  \begin{pmatrix}
h_1(z) - \bar{z}h_2(z)\\
h_2(z)-  zh_1(z)
\end{pmatrix}.

\end{array}$$ 

\noindent Thus $$ P_+\alpha = \frac{1}{2}  \begin{pmatrix}
h_1 - S^* h_2\\
-  S h_1 + h_2
\end{pmatrix},$$ where $S, S^*$ denote the shift and the backward shift operators on $H^2(\Dd,\Cc)$ respectively. Hence

\begin{equation}\label{proshift}
C_\xi^* C_\xi = \displaystyle\frac{1}{2} \begin{pmatrix}
1 & -S^* \\ -S &1
\end{pmatrix}.
\end{equation}

\noindent By equation (\ref{proshift}), 
 $$ \begin{array}{cllllll}

(C_\xi^* C_\xi)^2 &= \displaystyle \frac{1}{4}
\begin{pmatrix}
1 + S^*S & -2S^* \\ -2S & SS^*+1
\end{pmatrix}\vspace{2ex} \\ 

&= \displaystyle \frac{1}{4}\begin{pmatrix}
2 & -2S^* \\ -2S & 2- P_0 
\end{pmatrix}\vspace{2ex} \\ 

&= \displaystyle \frac{1}{2}\begin{pmatrix}
1 & -S^* \\ -S & 1-\frac{1}{2} P_0 
\end{pmatrix}\neq C_\xi^* C_\xi,

\end{array}$$ since $SS^*=1-P_0,$ where 
$P_0 (\sum_{n=0}^\infty a_n z^n) = a_0.$

\noindent Consequently $C_\xi^* C_\xi$ is not a projection and hence $C_\xi$ is not a partial isometry on $H^2(\Dd,\Cc^2)$. 
}
\end{example}

\section{Declarations}

 Chiotis was supported by a PhD studentship from the School of Mathematics, Statistics and Physics of  Newcastle University.\\
Lykova and Young were partially supported by the Engineering and Physical Sciences grant EP/N03242X/1.\\

We are very grateful to a referee for numerous improvements of our first version of the paper and for suggestions of missing references.\\

{\em Data availability statement:}  there are no data associated with this paper.\\

\end{document}